\documentclass[12pt]{amsart}
\usepackage{amssymb,amsmath,amsthm,newlfont,enumerate}

\theoremstyle{plain}
\newtheorem{theorem}{Theorem}[section]

\newtheorem{corollary}[theorem]{Corollary}
\newtheorem{problem}[theorem]{Problem}

\theoremstyle{remark}

\newcommand{\CC}{{\mathbb C}}
\newcommand{\DD}{{\mathbb D}}

\newcommand{\TT}{{\mathbb T}}
\newcommand{\cB}{{\mathcal B}}
\newcommand{\cD}{{\mathcal D}}
\newcommand{\cH}{{\mathcal H}}
\newcommand{\cM}{{\mathcal M}}

\DeclareMathOperator{\hol}{\mathrm Hol}

\DeclareMathOperator{\vmoa}{\mathrm VMOA}

\begin{document}

\title[Gleason--Kahane--\.Zelazko theorems]{Gleason--Kahane--\.Zelazko theorems in function spaces}

\author[Mashreghi]{Javad Mashreghi}
\address{D\'epartement de math\'ematiques et de statistique, Universit\'e Laval, 
Qu\'ebec City (Qu\'ebec),  Canada G1V 0A6}
\email{javad.mashreghi@mat.ulaval.ca}
\thanks{JM supported by an NSERC grant}

\author[Ransford]{Thomas Ransford}
\address{D\'epartement de math\'ematiques et de statistique, Universit\'e Laval, 
Qu\'ebec City (Qu\'ebec),  Canada G1V 0A6}
\email{thomas.ransford@mat.ulaval.ca}
\thanks{TR supported by grants from NSERC and the Canada Research Chairs program}

\date{24 Jan 2018}

\begin{abstract}
The Gleason--Kahane--\.Zelazko theorem states that a linear functional on a Banach algebra
that is non-zero on invertible elements is necessarily a scalar multiple of a character.
Recently this theorem has been extended to certain Banach function spaces that are not algebras.
In this article we present a brief survey of these extensions.
\end{abstract}



\maketitle


\section{Introduction}\label{S:intro}

The following result is known as the Gleason--Kahane--\.Zelazko theorem
(or GKZ-theorem for short).
For commutative algebras it was proved independently 
by Gleason \cite{Gl67} and by  Kahane and \.Zelazko \cite{KZ68}.
Subsequently \.Zelazko \cite{Ze68}  extended it to the  non-commutative case.

\begin{theorem}\label{T:GKZ}
Let $A$ be a complex unital Banach algebra, 
and let \mbox{$\Lambda:A\to\CC$} be a linear functional such that
$\Lambda(1)=1$ and $\Lambda(a)\ne0$ for all invertible elements $a\in A$. 
Then $\Lambda(ab)=\Lambda(a)\Lambda(b)$ for all $a,b\in A$.
\end{theorem}

Note that continuity of $\Lambda$ is not assumed.
This result is thus an early example of an automatic continuity theorem.

The original proofs of Theorem~\ref{T:GKZ} used results about entire functions. 
In \S\ref{S:GKZproof} we include  a completely elementary proof, due to Roitman and Sternfeld,
which we believe deserves to be more widely known.

The GKZ-theorem has been extended in many different ways; see for example the article of
Jarosz \cite{Ja91}.
Here we present a short survey of some recent extensions of the theorem
to function spaces.
The starting point is a generalization of the GKZ-theorem to 
$A$-modules, whose details are described in \S\ref{S:modules}. 
The module theorem leads to versions of the GKZ-theorem on  Hardy spaces and 
on weighted Dirichlet spaces. These are treated in 
\S\S\ref{S:Hardy}--\ref{S:Dirichlet}. 
In \S\ref{S:functspaces} we derive a GKZ-theorem in more general functions spaces,
but at the cost of assuming continuity of the linear functional.
Finally, in \S\ref{S:openproblems}, we conclude  with some open problems.


\section{An elementary proof of the GKZ-theorem}\label{S:GKZproof}

The following proof is due to Roitman and Sternfeld \cite{RS81}.

\begin{proof}[Proof of Theorem~\ref{T:GKZ}]
Let $a\in A$. Given $n\ge1$ and $t\in\CC$, we have
\[
\Lambda((t1-a)^n))=t^n-n\Lambda(a)t^{n-1}+\frac{1}{2}n(n-1)\Lambda(a^2)t^{n-2}-\dots +(-1)^n\Lambda(a^n).
\]
This is a polynomial in $t$ of degree $n$.
If $t>\|a\|$, then $(t1-a)^n$ is invertible,
and consequently $\Lambda((t1-a)^n)\ne0$. 
Thus the roots $\lambda_1,\dots,\lambda_n$ of the polynomial all satisfy $|\lambda_j|\le \|a\|$.
Now, by elementary algebra,
\[
(n\Lambda(a))^2-n(n-1)\Lambda(a^2)
=\bigl(\sum_{j=1}^n\lambda_j\bigr)^2-\sum_{\substack{j,k=1\\j\ne k}}^n\lambda_j\lambda_k
=\sum_{j=1}^n\lambda_j^2,
\]
whence
\[
|(n\Lambda(a))^2-n(n-1)\Lambda(a^2)|\le n\|a\|^2.
\]
Dividing through by $n^2$ and letting $n\to\infty$, we obtain
\begin{equation}\label{E:a^2}
\Lambda(a^2)=\Lambda(a)^2.
\end{equation}
Replacing $a$  by $a+b$ in this identity, we deduce that
\begin{equation}\label{E:ab+ba}
\Lambda(ab+ba)=2\Lambda(a)\Lambda(b).
\end{equation}
In particular, if $A$ is commutative, then $\Lambda(ab)=\Lambda(a)\Lambda(b)$,
proving the theorem in this case.

When $A$ is non-commutative, we use a variant of the argument of \.Zelazko 
due to Allan Sinclair.
Suppose, if possible, that there exist $a,b\in A$ such that $\Lambda(ab)\ne\Lambda(a)\Lambda(b)$.
Replacing $a$ by $\lambda a+\mu 1$, where $\lambda,\mu$ are suitably chosen scalars, we may suppose
that $\Lambda(a)=0$ and $\Lambda(ab)=1$. By \eqref{E:ab+ba}, it then follows that $\Lambda(ba)=-1$.
Set $c:=bab$. Then, on the one hand, we have 
\[
\Lambda(a)\Lambda(c)=0\Lambda(c)=0,
\]
whilst on the other hand, using \eqref{E:ab+ba} and then \eqref{E:a^2}, we have
\[
2\Lambda(a)\Lambda(c)=\Lambda(ac+ca)=\Lambda((ab)^2)+\Lambda((ba)^2)=\Lambda(ab)^2+\Lambda(ba)^2=2.
\]
This contradiction establishes the result.
\end{proof}


\section{A GKZ-theorem for modules}\label{S:modules}

The following result, which is a generalization of classical GKZ-theorem to modules,
was obtained in \cite{MR15}.

\begin{theorem}\label{T:module}
Let $A$ be a  complex unital Banach algebra, let $M$ be a left $A$-module,
and let $S$ be a non-empty subset of $M$ satisfying the following conditions:
\begin{enumerate}[\upshape(S1)]
\item $S$ generates $M$ as an $A$-module;
\item if $a\in A$ is invertible and $s\in S$, then $a s\in S$;
\item if $s_1,s_2\in S$, then there exist $a_1,a_2\in A$ such that 
$a_1 S\cup a_2S\subset S$ and $a_1 s_1=a_2 s_2$.
\end{enumerate}
Let $\Lambda:M\to\CC$ be a linear functional such that 
$\Lambda(s)\ne0$ for all $s\in S$. 
Then there exists a unique character $\chi$ on $A$ such that
\begin{equation}\label{E:lambdachi}
\Lambda(a m)=\chi(a)\Lambda(m) \qquad(a\in A,~m\in M).
\end{equation}
\end{theorem}

Notice that continuity of $\Lambda$ is not assumed.
Indeed, the $A$-module $M$ is not assumed to carry any topological structure.

Theorem~\ref{T:module} contains the classical GKZ-theorem as a special case,
as can be seen by taking $M=A$ and $S$ to be the set of invertible elements of $A$. 
Note, however, that the GKZ-theorem is used in its proof.

\begin{proof}
Uniqueness of $\chi$ is clear. 
Indeed, fixing any $s\in S$, by \eqref{E:lambdachi} we must have
\[
\chi(a)=\Lambda(a s)/\Lambda(s) \qquad(a\in A).
\]

To prove existence, we begin by deriving inspiration from this last equation. 
Given $s\in S$,  define $\chi_s:A\to\CC$ by
\[
\chi_s(a):=\Lambda(a s)/\Lambda(s) \qquad(a\in A).
\]
Clearly $\chi_s$ is a linear functional on $A$ satisfying $\chi_s(1)=1$, 
and from property~(S2) we have $\chi_s(a)\ne0$ for all invertible $a\in A$.
By Theorem~\ref{T:GKZ}, it follows that $\chi_s$ is a character on $A$.

Next, given $s_1,s_2\in S$,  property~(S3)
yields the existence of elements $a_1,a_2\in A$ such that 
$a_1S\cup a_2S\subset S$ and $a_1 s_1=a_2 s_2$. 
Then $\Lambda(a_1 s_1)=\Lambda(a_2 s_2)$, whence
\begin{equation}\label{E:chieqn1}
\chi_{s_1}(a_1)\Lambda(s_1)=\chi_{s_2}(a_2)\Lambda(s_2).
\end{equation}
Equally, for each $a\in A$, we have $aa_1 s_1=aa_2 s_2$, whence
\begin{equation}\label{E:chieqn2}
\chi_{s_1}(aa_1)\Lambda(s_1)=\chi_{s_2}(aa_2)\Lambda(s_2).
\end{equation}
Now both sides of \eqref{E:chieqn1} are non-zero, because $a_j s_j\in S$.
Thus we may divide \eqref{E:chieqn2} by \eqref{E:chieqn1} to obtain
\[
\chi_{s_1}(a)=\chi_{s_2}(a).
\]
In other words, $\chi_s$ is independent of $s$. 
Let us call it simply $\chi$. 
Note that we then have
\[
\Lambda(a s)=\chi(a)\Lambda(s) \qquad(a\in A,~ s\in S).
\]

Finally, let $a\in A$ and $m\in M$. 
By property~(S1), there exist $a_1,\dots,a_n\in A$ and $s_1,\dots,s_n\in S$ 
such that $m=\sum_1^na_j s_j$. Then we have
\begin{align*}
\Lambda(a m)
&=\sum_1^n\Lambda(aa_j s_j)
=\sum_1^n \chi(aa_j)\Lambda(s_j)\\
&=\chi(a)\sum_1^n \chi(a_j)\Lambda(s_j)
=\chi(a)\Lambda(m),
\end{align*}
which gives \eqref{E:lambdachi}.
\end{proof}

In practice, condition (S2) is usually easy to check, whereas (S1) and (S3) often require more effort.
If either (S1) or (S3) fails to hold, then it is still possible to apply Theorem~\ref{T:module} by proceeding as follows.
Let $A,M,S$ be as above and assume that (S2) holds, but not necessarily (S1) or (S3). 
Define a relation $\sim$ on $S$ by
\[
s_1\sim s_2 \iff \exists\, a_1,a_2\in A \text{~such that~} a_1S\cup a_2S\subset S \text{~and~}a_1s_1=a_2s_2.
\]
It is easy to check that $\sim$ is an equivalence relation on $S$. 
Let $S_0$ be an equivalence class of $S$ under this relation and let $M_0$
be the $A$-submodule generated by $S_0$. Then the triple $A,M_0,S_0$ satisfies all three conditions
(S1), (S2) and (S3), so Theorem~\ref{T:module} is applicable, and we deduce that there exists a character $\chi_0$ on $A$ (depending on the choice of $S_0$) such that
\[
\Lambda(am)=\chi_0(a)\Lambda(m) \quad(a\in A,~m\in M_0).
\]


\section{GKZ-theorems in Hardy  spaces}\label{S:Hardy}

Let $\DD$ denote the open unit disk and $\TT$ denote the unit circle. 
We write $\hol(\DD)$ for the space of holomorphic functions on $\DD$.
The Hardy spaces on $\DD$ are defined as follows:
\begin{align*}
H^p
&:=\Bigl\{f\in\hol(\DD):\sup_{r<1}\int_0^{2\pi}|f(re^{i\theta})|^p\,d\theta<\infty\Bigr\}
\quad(0<p<\infty),\\
H^\infty
&:=\Bigl\{f\in\hol(\DD):\sup_{z\in\DD}|f(z)|<\infty\Bigr\}.
\end{align*}
We say that $h\in \hol(\DD)$ is \emph{inner} if 
\[
|h(z)|\le 1 ~(z\in\DD)
\quad\text{and}\quad
\lim_{r\to1^-}|h(re^{i\theta})|=1 \textrm{~a.e.\ on $\TT$}.
\]
Also $g\in\hol(\DD)$ is \emph{outer} if there exists
$G:\TT\to[0,\infty)$ with $\log G\in L^1(\TT)$ such that
\[
g(z)=\exp\Bigl(\int_0^{2\pi}\frac{e^{i\theta}+z}{e^{i\theta}-z}
\log G(e^{i\theta})\,\frac{d\theta}{2\pi}\Bigr)\
\quad(z\in\DD).
\]
In this case, $g\in H^p$ if and only if $G\in L^p(\TT)$. 
It is a fundamental result that every $f\in H^p$ can be factorized in an essentially unique way
as $f=hg$, where $h$ is inner and $g\in H^p$ is outer.
For this and further background on Hardy spaces, we refer to Duren's book \cite{Du70}.

The following theorem, first obtained in \cite{MR15}, is an analogue of the GKZ-theorem for Hardy spaces.

\begin{theorem}\label{T:Hardy}
Let $0< p\le \infty$ and let $\Lambda:H^p\to\CC$ be a linear functional  
such that $\Lambda(1)=1$ and $\Lambda(g)\ne0$ for all outer functions $g\in H^p$.
Then there exists $w\in \DD$ such that
\[
\Lambda(f)=f(w) \qquad(f\in H^p).
\]
\end{theorem}

Note that, once again,  continuity of $\Lambda$  is not assumed.

\begin{proof}
We apply Theorem~\ref{T:module} with $M:=H^p$ and $A:=H^\infty$, 
taking  $S$ to be the set of outer functions in $H^p$.

Property~(S1) holds by the inner-outer factorization theorem for $H^p$.

Property~(S2) holds because every invertible function in $h\in H^\infty$ is outer.
Indeed, multiplying together  the inner-outer factorizations of $h$ and $1/h$, 
we obtain a factorization of $1$, and by uniqueness 
it follows  that the inner factors of $h$ and $1/h$  must both be  unimodular constants.

Property~(S3) holds because every outer function can be represented 
as the quotient of two bounded outer functions \cite[Proof of Theorem~2.1]{Du70}. 

Thus, by Theorem~\ref{T:module}, 
there exists a character $\chi$ on $H^\infty$ such that
\begin{equation}\label{E:Hplinfun}
\Lambda(hf)=\chi(h)\Lambda(f) \qquad(f\in H^p,~h\in H^\infty).
\end{equation}

Let $w:=\chi(u)$ 
(where $u$ denotes the function $u(z):=z$).
For all $\lambda\in\CC\setminus\DD$, 
the function $(u-\lambda 1)$ is outer, 
so we have  $\Lambda(u-\lambda 1)\ne0$,
whence $\chi(u-\lambda 1)\ne0$ and  $w\ne\lambda$. 
In other words, $w\in\DD$.

To finish the proof, 
we show that $\Lambda(f)=f(w)$ for all $f\in H^p$.
Given $f\in H^p$,  let us define $k(z):=(f(z)-f(w))/(z-w)$. 
Then $k\in H^p$ and  $f=f(w)1+(u-w1)k$. 
Applying $\Lambda$ to both sides of this 
last identity and using \eqref{E:Hplinfun},
we obtain 
\[
\Lambda(f)=f(w)\Lambda(1)+\chi(u-w1)\Lambda(k)=f(w)1+0\Lambda(k)=f(w),
\]
as desired.
\end{proof}

As an application of Theorem~\ref{T:Hardy}, 
we obtain the following very simple characterization of weighted composition operators
on Hardy spaces.

\begin{corollary}\label{C:Hardy}
Let $0< p\le\infty$ and let $T:H^p\to\hol(\DD)$ be a linear map 
such that $(Tg)(z)\ne0$ for all outer functions $g\in H^p$ and all $z\in\DD$.
Then there exist holomorphic functions 
$\phi:\DD\to\DD$ and $\psi:\DD\to\CC\setminus\{0\}$ such that 
\[
Tf=\psi.(f\circ\phi)\qquad(f\in H^p).
\]
\end{corollary}

\begin{proof}
Set $\psi:=T1$. 
This is a  holomorphic function on $\DD$, 
and is nowhere zero because $1$ is outer.
Define $\phi:=(Tu)/\psi$, where $u(z):=z$. 
This too  is a holomorphic function  on $\DD$. 

Fix $z\in\DD$. The map $f\mapsto (Tf)(z)/\psi(z)$ 
is a linear functional on $H^p$ that sends $1$ to $1$ and is non-zero on outer functions. 
By Theorem~\ref{T:Hardy}, 
there exists $w\in\DD$ 
such that $(Tf)(z)/\psi(z)=f(w)$ for all $f\in H^p$. 
Taking $f:=u$, we see that $w=\phi(z)$. 
Thus  $\phi(z)\in\DD$ and $(Tf)(z)=\psi(z)f(\phi(z))$ for all $f\in H^p$.
\end{proof}


\section{Reproducing kernel Hilbert spaces}\label{S:RKHS}

The results in the previous section depend heavily on the inner-outer factorization
in Hardy spaces. If we want to extend them to other function spaces using the same basic method,
then we need to establish analogous factorization theorems for these spaces. In general this is difficult. 
But there is one family of spaces where, thanks to some recent developments, it is possible.
The family in question is a certain  class of reproducing kernel Hilbert spaces. 
In this section we describe the background
leading to the factorization theorem, and in the next section we apply it to deduce a GKZ-theorem
for Dirichlet spaces.

Let $\Omega$ be a set. A \emph{reproducing kernel Hilbert space} (RKHS) on $\Omega$ is
a Hilbert space $\cH$ of complex-valued functions on $\Omega$ such that, for each $t\in\Omega$,
the  evaluation functional $f\mapsto f(t):\cH\to\CC$ is continuous. 

Let $\cH$ be a RKHS on $\Omega$.
By the Riesz representation theorem, given $t\in\Omega$, there exists a unique $k_t\in\cH$ such that 
\[
f(t)=\langle f,k_t\rangle \quad(f\in\cH).
\]
The function $k_t$ is called the \emph{reproducing kernel for the point $t$}.
The function $K:\Omega\times\Omega\to\CC$ defined by $K(s,t):=k_t(s)=\langle k_t,k_s\rangle_\cH$ is called the \emph{reproducing kernel for $\cH$}.

A \emph{multiplier} for $\cH$ is a function $h:\Omega\to\CC$
such that $hf\in\cH$ for all $f\in\cH$. 
By the closed graph theorem, $f\mapsto hf$ is then a bounded operator on $\cH$.
The set of multipliers is denoted $\cM(\cH)$.
It is a Banach algebra with respect to the multiplier norm, defined by
\[
\|h\|_{\cM(\cH)}:=\sup\{\|hf\|_\cH:\|f\|_\cH\le1\}.
\]

A function $F:\Omega\times\Omega\to\CC$ is called \emph{positive semidefinite} if, 
for each choice of points $t_1,\dots,t_n\in \Omega$,
the $n\times n$ matrix $(F(t_i,t_j))$ is positive semidefinite. 
A reproducing kernel $K(s,t)$ is always positive semidefinite. 
An RKHS is said to have the  \emph{complete Pick property} if its kernel satisfies
$K(s,t)>0$ for all $s,t\in\Omega$
and if there exists a function $u:\Omega\to\CC\setminus\{0\}$
such that the map
\[
(s,t)\mapsto\Bigl(1-\frac{u(s)\overline{u(t)}}{K(s,t)}\Bigr)
\]
is positive semidefinite. 
(The terminology arises from a connection with the Pick interpolation problem,
which we shall not discuss here.)

For general background on RKHS's we refer to the book \cite{PR16},
and for more information on the complete Pick property we recommend \cite{AM02}.

Now we can state the promised factorization theorem.
It is due to to Aleman, Hartz, McCarthy and Richter \cite[Theorem 1]{AHMR17}.

\begin{theorem}\label{T:AHMR}
Let $\cH$ be an RKHS on $\Omega$.
Assume that $\cH$ has the complete Pick property. 
Then each $f\in\cH$ can be factorized as $f=h/k$, where 
$h,k\in\cM(\cH)$ and $k$ is nowhere-zero on $\Omega$.
\end{theorem}

Examples of RKHS's with the complete Pick property include the Hardy space $H^2$
and the Dirichlet space $\cD$. On the other hand, the Bergman space $A^2:=L^2(\DD)\cap\hol(\DD)$ is an
RKHS without the complete Pick property, and in fact Theorem~\ref{T:AHMR} fails for $A^2$.
Indeed, the multiplier algebra of $A^2$ is $H^\infty$, 
but not every function $f\in A^2$ can be written as $h/k$ with $h,k\in H^\infty$, 
because $h/k$ has radial limits a.e.\ on~$\TT$, yet there exist  $f\in A^2$ without any radial limits at all 
(see  \cite{Ca64} for a very simple example).


\section{GKZ-theorems in Dirichlet spaces}\label{S:Dirichlet}

Given a positive superharmonic function $\omega$ on $\DD$, 
the \emph{weighted Dirichlet space} $\cD_\omega$
is the set of $f\in\hol(\DD)$ such that
\[
\cD_\omega(f):=\frac{1}{\pi}\int_\DD |f'(z)|^2\omega(z)\,dA(z)<\infty.
\]
The weight $\omega$ is automatically integrable, so $\cD_\omega$ contains all polynomials.
One can show that $\cD_\omega\subset H^2$, 
and that $\cD_\omega$ becomes a Hilbert space when endowed with
the norm $\|\cdot\|_{\cD_\omega}$ defined by
\[
\|f\|_{\cD_\omega}^2:=\|f\|_{H^2}^2+\cD_\omega(f).
\]

Taking $\omega\equiv1$, we obtain the classical Dirichlet space $\cD$.
Other interesting examples include the standard weighted Dirichlet spaces $\cD_\alpha$ for $0<\alpha<1$
(obtained by taking $\omega(z):=(1-|z|^2)^\alpha$),
and the harmonically weighted Dirichlet spaces introduced by Richter in \cite{Ri91} 
and further studied by Richter and Sundberg in \cite{RS91}.
The study of Dirichlet spaces with general superharmonic weights 
was initiated by Aleman in his habilitation thesis \cite{Al93},
where further details on this subject may be found.

The following result, obtained in \cite{MRR17}, is an analogue of the GKZ-theorem for Dirichlet spaces.

\begin{theorem}\label{T:Dirichlet}
Let $\omega$ be a positive superharmonic function on $\DD$.
Let  $\Lambda:\cD_\omega\to\CC$ be a linear functional  
such that $\Lambda(1)=1$ and  $\Lambda(g)\ne0$ for all nowhere-vanishing functions $g\in \cD_\omega$.
Then there exists $w\in \DD$ such that
$\Lambda(f)=f(w)$ for all $f\in\cD_\omega$.
\end{theorem}

\begin{proof}
We apply Theorem~\ref{T:module}, 
taking $M$ to be $\cD_\omega$ and $A=\cM(\cD_\omega)$,  the multiplier algebra of $\cD_\omega$.
Let $S$  be  the set of nowhere-vanishing functions in $\cD_\omega$.
We need to check the conditions (S1), (S2) and (S3) of Theorem~\ref{T:module}.

For condition~(S1), it suffices to show that every $f\in\cD_\omega$ can be written as 
$f=g_1+g_2$, where $g_1,g_2\in\cD_\omega$ and neither function $g_j$ has a zero in $\DD$. 
Given $f\in\cD_w$, factor it as $f=hg$, where  $h$ is inner and $g$ outer.
By \cite[Chapter~IV, Theorem~3.4]{Al93}, we have $g\in\cD_w$.
Then  $f=(h-2)g+2g$ is the sum of two nowhere-vanishing 
functions in $\cD_w$. Thus condition~(1) is verified.

Condition (S2) is obviously satisfied, 
since invertible elements of $\cM(\cD)$ must be everywhere non-zero on $\DD$.

For condition~(S3), let $g_1,g_2$ be nowhere-vanishing elements of $\cD_\omega$.
By a theorem of Shimorin \cite{Sh02}, 
the space $\cD_\omega$ is a RKHS with the complete Pick property.
Therefore, by Theorem~\ref{T:AHMR}, we can write  $g_j=h_j/k_j$, 
where $h_1,k_1,h_2,k_2$ are nowhere-vanishing elements of $\cM(\cD_\omega)$. 
Set $a_1:=h_2k_1$ and $a_2:=h_1k_2$. 
These $a_1,a_2$ are nowhere-vanishing elements of $\cM(\cD_\omega)$,
so $a_jS\cup a_2S\subset S$, and further $a_1g_1=a_2g_2$. 
Thus condition~(S3) is satisfied.

By Theorem~\ref{T:module}, 
there exists a character $\chi$ on $\cM(\cD_\omega)$ such that
\begin{equation}\label{E:Dlinfun}
\Lambda(hf)=\chi(h)\Lambda(f) \qquad(f\in \cD,~h\in \cM(\cD_\omega)).
\end{equation}
The rest of the proof is just the same as in the proof of Theorem~\ref{T:Hardy},
replacing $H^p$ by $\cD_\omega$ throughout.
\end{proof}

Just as before, we deduce a characterization of weighted composition operators,
this time on the spaces $\cD_\omega$. 

\begin{corollary}\label{C:Diriclet}
Let $\omega$ be a positive superharmonic function on $\DD$,
and let $T:\cD_\omega\to\hol(\DD)$ be a linear map 
such that $(Tg)(z)\ne0$ for all nowhere-vanishing functions $g\in\cD_\omega$ and all $z\in\DD$.
Then there exist holomorphic functions 
$\phi:\DD\to\DD$ and $\psi:\DD\to\CC\setminus\{0\}$ such that 
\[
Tf=\psi.(f\circ\phi)\qquad(f\in \cD_\omega).
\]
\end{corollary}


\section{GKZ-theorems in other function spaces}\label{S:functspaces}

The proofs of Theorems~\ref{T:Hardy} and \ref{T:Dirichlet} use special factorization properties of Hardy spaces
and Dirichlet spaces,
and do not easily generalize to other families of spaces. 
However, if we are willing to assume the continuity of the  linear maps involved, 
then these theorems
do  extend to a wide variety of other spaces, 
albeit with a  different proof.

In what follows, 
we shall consider a Banach space $X\subset\hol(\DD)$ with the following properties:
\begin{enumerate}[(X1)]
\item for each $w\in\DD$, the  map $f\mapsto f(w):X\to\CC$ is continuous;
\item $X$ contains the polynomials and they are dense in $X$;
\item $X$ is shift-invariant: $f\in X\Rightarrow zf\in X$.
\end{enumerate}
We write $\cM(X)$ for the \emph{multiplier algebra} of $X$, namely
\begin{align*}
\cM(X)&:=\{h\in\hol(\DD): hf\in X \text{~for all~} f\in X\},\\
\|h\|_{\cM(X)}&:=\sup\{\|hf\|_X:\|f\|_X\le1\}.
\end{align*}
Using property (X1) above, it is not hard to see that
$\cM(X)$ can be identified with a closed subalgebra of 
the algebra of all bounded linear operators on $X$, so it is a Banach algebra.
For each $w\in\DD$, 
the evaluation functional $h\mapsto h(w)$ is a character on $\cM(X)$,
so $|h(w)|\le\|h\|_{\cM(X)}$. 
It follows that $\cM(X)\subset H^\infty$.
As $X$ contains the constants,  we also have $\cM(X)\subset X$.

We also consider a subset $Y$ of $X$ with the following properties:
\begin{enumerate}[(Y1)]
\item if $g\in X$ and $0<\inf_\DD|g|\le\sup_\DD|g|<\infty$, then $g\in Y$;
\item if $g(z):=z-\lambda$ where $\lambda\in\TT$, then $g\in Y$.
\end{enumerate}
 
Examples of spaces $X$ satisfying (X1)--(X3) include:
\begin{itemize}
\item the Hardy spaces $H^p~(1\le p<\infty)$;
\item the Bergman spaces $A^p~(1\le p<\infty)$;
\item the holomorphic Besov spaces $B_p~(1\le p<\infty)$;
\item the superharmonically weighted Dirichlet spaces $\cD_\omega$;
\item the  Sobolev spaces $S^p:=\{f:f'\in H^p\} ~(1\le p<\infty)$;
\item the disk algebra $A(\DD)$;
\item the little Bloch space $\cB_0$;
\item the  space $\vmoa$ of functions of vanishing mean oscillation;
\item the de Branges--Rovnyak spaces $\cH(b)$, for non-extreme points $b$  in the unit ball of $H^\infty$.
\end{itemize}

Examples of sets $Y$ satisfying (Y1)--(Y2) include:
\begin{itemize}
\item the nowhere-zero functions in $X$;
\item the outer functions in $X$;
\item the cyclic functions, if $X=H^p,A^p, B_p,\cD_\omega,\cB_0$ or $\vmoa$.
\end{itemize}

For background on these various spaces, we refer to \cite{Al93, EKMR14, FM16, Zh07}.

The following result is an analogue of the GKZ-theorem in this context.
Note that, this time, we \emph{do} assume that $\Lambda$ is continuous.

\begin{theorem}\label{T:linfun}
Let $X\subset\hol(\DD)$ be a Banach space satisfying (X1)--(X3) above.
Let $Y\subset X$ be a set satisfying (Y1)--(Y2) above.
Let $\Lambda:X\to\CC$ be a continuous linear functional such that 
$\Lambda(1)=1$ and $\Lambda(g)\ne0$ for all $g\in Y$.
Then there exists $w\in\DD$ such that 
\[
\Lambda(f)=f(w) \qquad(f\in X).
\]
\end{theorem}

\begin{proof}
If $h$ is invertible in $\cM(X)$, 
then  $h$ and $1/h$ belong to $H^\infty\cap X$, 
so $h\in Y$, and $\Lambda(h)\ne0$.
By Theorem~\ref{T:GKZ}, $\Lambda$ is a character on $\cM(X)$. 
As $X$ is shift-invariant, we have $u\in\cM(X)$ (where $u(z):=z$).
Set $w:=\Lambda(u)$. 
Then $\Lambda(p)=p(w)$ for all polynomials $p$.
If $|\lambda|\ge1$ then $u-\lambda1\in Y$, 
so $\Lambda(u-\lambda1)\ne0$. 
Consequently $w\in\DD$.
By (X1),
the evaluation functional  $f\mapsto f(w)$ is continuous on $X$. 
As polynomials are dense in $X$, 
we conclude that $\Lambda(f)=f(w)$ for all $f\in X$.
\end{proof}

Just as before, we deduce a characterization of weighted composition operators.
We endow $\hol(\DD)$ with its usual Fr\'echet-space topology.

\begin{corollary}\label{C:linmap}
Let $X\subset\hol(\DD)$ be a Banach space satisfying (X1)--(X3) above.
Let $Y\subset X$ be a set satisfying (Y1)--(Y2) above.
Let $T:X\to \hol(\DD)$ be a continuous linear map such that 
$Tg(z)\ne0$ for all $g\in Y$ and all $z\in\DD$. 
Then there exist holomorphic functions 
$\phi:\DD\to\DD$ and $\psi:\DD\to\CC\setminus\{0\}$ such that
\[
Tf=\psi.(f\circ\phi) \qquad (f\in X).
\]
\end{corollary}

In particular, this result applies to the linear maps $T:X\to\hol(\DD)$ that preserve outer functions.
Curiously, the  linear $T:X\to\hol(\DD)$ that preserve \emph{inner} functions can also be characterized as weighted composition operators. This is a special case of the next result.

\begin{theorem}\label{T:inner}
Let $X\subset\hol(\DD)$ be a Banach space satisfying (X1)--(X3) above.
Assume in addition  that the function $u(z):=z$ satisfies $\limsup_{n\to\infty}\|u^n\|_X^{1/n}\le1$.
Let $T:X\to\hol(\DD)$ be a continuous linear map that maps finite Blaschke products to inner functions,
and suppose further that $\dim(T(X))>1$.
Then there exist inner functions $\phi,\psi$, with $\phi$ non-constant, such that
\[
Tf=\psi.(f\circ \phi) \quad(f\in X).
\]
\end{theorem}

We omit the proof, which is different in spirit from that of all the preceding results.
The details can be found in \cite{MR17}.


\section{Open problems}\label{S:openproblems}

The results in this short survey lead to open problems in other fields.
We give two examples.

In Theorem~\ref{T:linfun}, we assumed the continuity of the linear functional~$\Lambda$.
This naturally raises  a question about automatic continuity.

\begin{problem}
Let $X,Y$ be as in Theorem~\ref{T:linfun} and let $\Lambda:X\to\CC$ be a linear functional such that $\Lambda(g)\ne0$ for all $g\in Y$. Is $\Lambda$ automatically continuous?
\end{problem}

Thanks to Theorems~\ref{T:Hardy} and \ref{T:Dirichlet}, we know that the answer
is yes if $(X,Y)=(H^p,\text{outer functions})$ or $(X,Y)=(\cD_\omega,\text{non-zero functions})$. 
However, the proofs use special properties of these spaces.
Is there is an abstract automatic continuity argument that takes care of general $(X,Y)$?

The proof of Theorem~\ref{T:Dirichlet} used two factorization properties of $\cD_\omega$, namely,
every $f\in\cD_\omega$ can be written as 
\begin{itemize}
\item $f=h/k$, where $h,k\in\cM(\cD_\omega)$, and
\item $f=bg$, where  $b$ is bounded and $g\in\cD_\omega$ is zero-free.
\end{itemize}
The first of these is a consequence of a general theorem about RKHS's (Theorem~\ref{T:AHMR}).
However, the second appears to make use of special properties of Dirichlet spaces.
This prompts the following series of questions.

\begin{problem}
Let $\cH$ be a RKHS with the complete Pick property, and let $f\in\cH$.
\begin{enumerate}[(1)]
\item Can we write $f=bg$, where $b$ is bounded and $g\in\cH$ is zero-free?
\item Can we write $f=hg$, where $h\in\cM(\cH)$ and $g\in\cH$ is zero-free?
\item Can we write $f=hg$, where $h\in\cM(\cH)$ and $g\in\cH$ is cyclic?
\end{enumerate}
(To say that $g$ is \emph{cyclic} means that $\overline{\cM(\cH)g}=\cH$.)
\end{problem}

An affirmative answer to (2) or (3) would likely have applications in other fields.


\end{document}